\documentclass[letterpaper, 10 pt, conference]{ieeeconf}
\usepackage[utf8]{inputenc}
\usepackage{amsmath}
\usepackage{amssymb}
 
\usepackage{amsthm}
\usepackage{graphicx}
\usepackage{xcolor}
\usepackage{caption}
\usepackage{subcaption}
\usepackage{balance}
\usepackage{cite}

\IEEEoverridecommandlockouts                         
\newtheorem{thm}{Theorem}
\newtheorem{lem}[thm]{Lemma}

\newtheorem{exmp}[thm]{Example}

\definecolor{cadmiumred}{rgb}{0.89, 0.0, 0.13}

\newcommand{\Real}{\mathbb{R}}

\newcommand{\norm}[1]{\left\lVert#1\right\rVert}

\title{\LARGE\textbf{A Modelling Framework for Energy-Management \\ and Eco-Driving Problems using Convex Relaxations}}
\author{Y.J.J. Heuts and M.C.F. Donkers
\thanks{
The authors are with the Department of Electrical Engineering, Eindhoven University of Technology, Netherlands
	{\tt\small \{y.j.j.heuts, m.c.f.donkers\}@tue.nl}}%
}

\begin{document}

\maketitle
\begin{abstract}
This paper presents a convex optimization framework for eco-driving and vehicle energy management problems. We will first show that several types of eco-driving and vehicle energy management problems can be modelled using the same notions of energy storage buffers and energy storage converters that are connected to a power network. It will be shown that these problems can be formulated as optimization problems with linear cost functions and linear dynamics, and nonlinear constraints representing the power converters. We will show that under some mild conditions, the (non-convex) optimization problem has the same (globally) optimal solution as a convex relaxation. This means that the problems can be solved efficiently and that the solution is guaranteed to be globally optimal. Finally, a numerical example of the eco-driving problem is used to illustrate this claim.
\end{abstract}

\section{Introduction}
Improving vehicle energy efficiency is an important topic of research for the automotive industry in order to reduce the carbon footprint of vehicles as well as increase the vehicle's range per charge. These goals can be obtained by applying optimization methods, such as eco-driving and energy management. The eco-driving problem, see, e.g., \cite{padilla2018global,borsboom2021convex, ebbesen2017time,johannesson2015look}, uses a longitudinal vehicle model and information on the route to predict the most energy or time optimal velocity profile in order to reach the destination and can be implemented to actively coach the driver to drive in a more energy efficient way. Vehicle energy management, see, e.g., \cite{scordia2005global,schmid2021energy,jager2013HEV,romijn2018distributed,padilla2022complete,egardt2014electromobility,padilla2020port} attempts to optimize the energy flow in the vehicle to use the least amount of energy as possible. This leads to a decrease in operational cost, or an increase in driving range. Both methods are often at the basis of many current developments, such as platooning \cite{MA2021102746}, time-optimal control in traffic scenarios~\cite{hamednia2020time}, optimal component sizing \cite{HUANG2018132} and emission management for heavy-duty vehicles \cite{mennen2022sequential}. 

The optimization problems that arise in eco-driving and energy management may be solved in different ways, such as Dynamic Programming (DP), see, e.g., \cite{scordia2005global,johannesson2015look}, or Pontryagin's maximum principle (PMP), see, e.g., \cite{schmid2021energy,jager2013HEV}, but these methods can typically not appropriate when the complexity of the problem increases. A recent trend is therefore to consider static optimization to handle the complexity of the problem, which can then be solved by off-the-shelve solvers. High-fidelity powertrain models lead to nonlinear optimization problems that can be solved using derivative-free~methods, see, e.g., \cite{gao2007hybrid}, or by sequentially linearizing the problem, as is done in SQP, see, e.g., \cite{padilla2018global,mennen2022sequential}. These methods do not scale well in terms of number of components and do not always warrant optimality, as it is possible to end up in a local minimum. Therefore, problem formulations that lead to convex optimization problems are very desirable, as they scale well and lead to globally optimal solutions. 

Formulations of eco-driving and energy management that are either convex or use convex relaxations have been considered in the literature, e.g., in
\cite{romijn2018distributed,egardt2014electromobility,padilla2022complete,borsboom2021convex, ebbesen2017time}. The main difference between the approaches in these papers are the properties that are required for the component models. In \cite{egardt2014electromobility}, it is shown that models needs to be monotonous in one argument. On the other hand, it is argued in \cite{romijn2018distributed,padilla2022complete} that energy management problems can be modelled as power networks, where converters, such as the electric machine or internal combustion engine, are limited to quadratic functions. Finally, \cite{borsboom2021convex,ebbesen2017time} utilizes hyperbolic functions, such that they become second-order constraints when relaxed, though a formal proof on exactness of this relaxation is not given. Hence, convex modelling techniques are restrictive in the allowed component models and/or problem structure and  statements on exactness of the relaxation are not sufficiently general.

In this paper, we generalize the framework for complete vehicle energy management problems of \cite{romijn2018distributed,padilla2022complete}, and we will show that examples from \cite{padilla2018global,egardt2014electromobility,borsboom2021convex,ebbesen2017time,padilla2020port} also fit this framework. If the problem has been set up to fit the framework, it can be solved as a convex optimization problem by relaxing the (nonlinear) equalities representing the energy converter models into inequalities. We will prove that under some mild conditions on the network and the converter models, the solution of the relaxed problem will also be the global solution of the original non-convex problem. These conditions generalize the conditions of \cite{padilla2022complete} and \cite{egardt2014electromobility} warrant that the constraints are linearly independent, meaning that strong duality between the primal and dual problem holds. Then, using the dual optimization problem, it will be shown that the convex relaxation solution is equal to the non-convex solution and can therefore be used to solve the problem using more efficient tools. Moreover, we show that for several converter models, it is possible to formulate second-order conic constraints, for which a wide variety of solvers exist.

The remainder of this paper is organized as follows, Section~\ref{sec:modelling-framework} introduces the modelling framework, and gives examples of problems that fit within the framework. In Section~\ref{sec:solution-framework}, we will present the convex relaxation and give conditions under which the relaxation is exact. A numerical example is provided in Section~\ref{sec:numex} and the conclusions are drawn in Section~\ref{sec:conclusions}.

\section{Modelling framework}\label{sec:modelling-framework}
In this section, we will first discuss a description of a power network on which our framework is based. Subsequently, the optimal control problem is presented underlying the energy management problems. Finally, by introducing three examples, we show that the framework is general enough to capture all these problems.

\begin{figure}[t]
    \centering
    \vspace{2mm}
    \includegraphics[width=\linewidth]{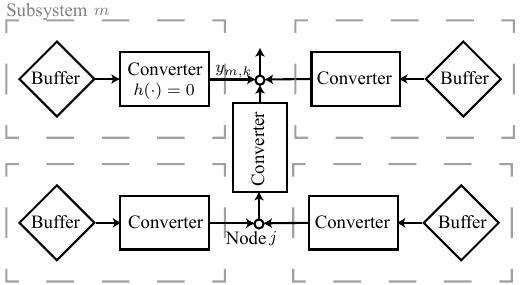}
    \caption{Schematic of a vehicle power network.}
    \label{fig:powernetwork}
\end{figure}

\subsection{General description}
Energy management problems aim to minimize the energy consumption for a network of subsystems $m\in \mathcal{M} := \{1 \dots M\}$, where $M$ is the total number of subsystems. These subsystems can be composed of buffers and converters, see Fig.~\ref{fig:powernetwork}. Buffers model the available stored energy within the network, e.g., the amount of charge in a battery. Converters, on the other hand, model the energy conversion from one domain to another, e.g., from chemical to mechanical energy. Furthermore, converters are connected to each other in a specific topology via nodes $j \in \mathcal{J} := \{1\dots J\}$, where $J$ indicates the number of nodes in the network. 

\subsection{Optimal Control Problem} \label{sec:canon}
The optimal control problem in the above power network aims to minimize the total energy usage of the components over a time horizon $k \in \{0\dots K-1\}$. A cost function adhering to this goal can be expressed as
\begin{subequations}\label{eq:nonconvex-opt}
    \begin{equation}
        \min_{y_{m,k},u_{m,k},x_{m,k}}  \sum_{k\in\mathcal{K}} \sum_{m\in \mathcal{K}} \!  a_m^\top u_{m,k} + b_m y_{m,k},
    \end{equation}
    where $a_m$ and $b_m$ are coefficients to define the cost function, and $u_{m,k}$ are the inputs and $y_{m,k}$ are scalar outputs. The minimization is subject to the input-output behaviour of the converters that are given by, 
    \begin{equation}
        h_m(x_{m,k},u_{m,k},y_{m,k}) = 0, \label{eq:nonconvex-opt-h}
    \end{equation}
    for all $m\in\mathcal{M}$ and all $k\in\mathcal{K}$, for which additional conditions will be introduced in the next section. When compared to \cite{romijn2018distributed,padilla2022complete}, the concept of energy conversion is broadened in this paper to include also the state $x_{m,k}$ and more general classes of functions. The converters are subject to the linear system dynamics of the energy buffer, given by
    \begin{equation}
        x_{k+1}=A x_{k} + B u_{k} + f_{k},
    \end{equation}
    where the $A$ and $B$ matrices are typically block diagonal (as the buffers are only connected though the converters), $x_k = [x_{1,k}^\top\dots x_{M,k}^\top]^\top$ is the state of all buffers, with $x_{0} = x_{\text{init}}$ are the initial states of the buffers, $u_k = [u_{1,k}^\top\dots u_{M,k}^\top]^\top$ is the inputs to all the buffers, and $f_{k}$ is a known disturbance to the system. In case the buffer is nonlinear, it can often be decomposed in a linear dynamic part (1c) and a static nonlinearity (1b), see the example in \cite{padilla2022complete}. In order to connect the buffers and the converters, we require power network nodes, which are given by
    \begin{equation}
        E x_{k} +  F u_{k} + G y_{k} + s_{k}= v_{k},        
    \end{equation}
    where $E$, $F$ and $G$ describe the network connections, $y_k = [y_{1,k}\dots y_{M,k}]^\top$, and $s_{k}\geqslant0$ is a slack variable which is used to differentiate between a conservative node in which all energy that passes through the node is conserved, leading to that element of $s_k$ being zero and a dissipative node, where part of the energy may be dissipated, i.e., through the brakes, leading to that element of $s_{k}$ to be non-negative. Furthermore, the load signal, $v_{k}$ is assumed to be known at each time instant $k\in\mathcal{K}$. Lastly, the states and inputs have to adhere to physical upper and lower bounds, i.e.,
    \begin{equation}
        \underline{x} \leqslant x_{k} \leqslant \overline{x}\quad \text{and} \quad \underline{u} \leqslant u_{k} \leqslant \overline{u}.
    \end{equation}
\end{subequations}

\subsection{Examples}
In order to show that a wide variety of energy management problems fit in this framework, examples are given of an eco-driving problem, and an energy management problem either defined using powers and energies or using a port-Hamiltonian representation.

\begin{exmp}\label{exmp:ED}
Eco-driving aims to find an optimal velocity profile and corresponding input power. Examples from literature are given in, e.g., ~\cite{padilla2018global, johannesson2015look}. A simplified longitudinal vehicle model is considered, and this model can be given as function of time or travelled distance. The longitudinal vehicle model is given by Newton's second law
\begin{equation}
m_e \tfrac{\mathrm{d}v}{\mathrm{d}t} = F_p - c_gv^2 - mg c_r\cos \alpha, - mg\sin \alpha, \label{eq:Newton_s_time}
\end{equation}
where $v$ is the longitudinal velocity, $c_g$ is the lumped aerodynamic drag constant and $c_r$ is the rolling resistance constant, $F_p$ is the propulsion force defined by the brake (br) force and electric machine (em) force $F_p = F_{br} + F_{em}$, $g$ is the gravitational acceleration, $m_e$ is the equivalent vehicle mass including the actual mass $m$, and estimated induced inertia by rotational parts to make the model linear. The road gradient is indicated by $\alpha$. The vehicle's kinetic energy is given by
\begin{equation}
E_k = \tfrac{1}{2}m_ev^2. \label{eq:Ekin}
\end{equation}
We can introduce the coordinate transformation from the time to distance domain
\begin{equation}
\tfrac{\mathrm{d}v}{\mathrm{d}t} = v\tfrac{\mathrm{d}v}{\mathrm{d}s} = \tfrac{1}{2}\tfrac{\mathrm{d}}{\mathrm{d}s}v^2, \label{eq:vdvds}
\end{equation}
which leads to
\begin{equation}
\tfrac{\mathrm{d}}{\mathrm{d}s}E_k = F_p - mg(\sin \alpha(s)\! +\! c_r\cos \alpha(s)) - \tfrac{c_g}{m_e}E_k. \label{eq:Newton_s_distance}
\end{equation}
In addition, we need the concept of lethargy, $\ell=\frac{dt}{ds}$, which satisfies $v\frac{dt}{ds} = 1$ and is used to keep time in our problem formulation. The complete block diagram of the eco-driving problem can be seen in Fig.~\ref{fig:blockdiagramecodriving}.

	\begin{figure}[t]
		\centering
		\includegraphics[width=1\linewidth]{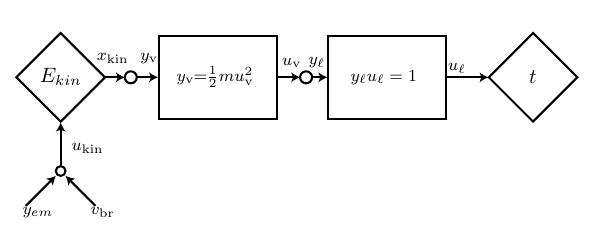}
		\caption[Block diagram of eco-driving problem]{Block diagram of eco-driving problem}
		\label{fig:blockdiagramecodriving}
	\end{figure}
 
	The discretized problem with time constraint can then be formulated as
	\begin{subequations}
		\begin{align}
			&\min \sum_{k\in\mathcal{K}}  y_{em,k}\\
			\text{s.t.} &\quad y_{\textrm{v},k} = 0.5mu_{\textrm{v},k}^2 \label{eq:mv2} \\
			&\quad u_{\ell,k}y_{\ell,k} = 1\label{eq:lethv}\\
            &\quad x_{\textrm{kin},k+1} = A_d x_{\textrm{kin},k} + B_du_{\textrm{kin},k} + u_{\textrm{kin},k} \label{eq:eco-dyn1}\\
			&\quad x_{\ell,k+1} = x_{\ell,k} + \delta_s\cdot u_{\ell,k} \label{eq:eco-dyn2} \\ 
			&\quad u_{kin,k} + y_{\textrm{em},k} + s_{\textrm{br},k} = 0 \label{eq:eco-netw1}\\
            &\quad y_{\ell,k} + u_{\textrm{v},k} = 0 \label{eq:eco-netw2} \\ 
            &\quad y_{\textrm{v},k} - x_{\textrm{kin},k} = 0 \label{eq:eco-netw3} \\
			&\quad x_{t,K} \leqslant T_{\max},
		\end{align}
	\end{subequations}
	with $A_d = e^{-c_g\delta_s/m_e}$, $B_d = \frac{m_e}{c_g}(1 - A_d)$, and $K$ being the horizon length, $\delta_s$ being the sample distance and the lethargy $u_{\ell,k} = \frac{ds}{dt}$. Eqns. \eqref{eq:mv2} and \eqref{eq:lethv} describe the converters, \eqref{eq:eco-dyn1} and \eqref{eq:eco-dyn2} cover the dynamics of the problem, and \eqref{eq:eco-netw1}, \eqref{eq:eco-netw2} and \eqref{eq:eco-netw3} describe the power network.
	
\end{exmp}
\begin{exmp}\label{exmp:EM}
	Vehicle energy management has the goal to distribute the vehicle's energy requirements, resulting from propelling the vehicle, as well as the vehicle's auxiliary systems, optimally over one or multiple energy sources (i.e., internal combustion engine, battery, or fuel cell) \cite{jager2013HEV}. This problem can be extended to a `complete vehicle energy management' problem by adding auxiliary vehicle subsystems to the topology in order to optimize the energy flow for all consumers and generators inside the vehicle \cite{romijn2018distributed}. 
 
    One possible power network is shown in Fig.~\ref{fig:blockdiagramseries-hybrid-powernetwork}. In this example, we consider two energy buffers, namely a fuel-based converter, such as an internal combustion engine, and a battery, one node and three converters. The buffers can be modelled as
	\begin{equation}
		\dot E_j = -P_j, \quad m \in \{f, s\},
	\end{equation}
	where $f$ denotes fuel and $s$ the stored battery energy. Both buffers are upper and lower bounded, and the flow of fossil fuel can only be positive, i.e., $P_f > 0$. The converters are given by quadratic relations in \cite{padilla2022complete}, such that
	\begin{equation}
		y_{m,k} =  \alpha_{2,m} u_{m,k}^2 + \alpha_{1,m} u_{m,k} + \alpha_{0,m} 
	\end{equation}
    for $m \in \{\text{f}, \text{bat}, \text{em}\}$, where $u_{m,k}$ and $y_{m,k}$ denote the input and output of the converter. The energy balance within a node is given by
	\begin{equation}
		y_{\textrm{f}} + y_{\textrm{em}} = v_p + u_{\textrm{br}} \quad \text{with} \quad u_{\textrm{br}}\geqslant0
	\end{equation}
    where $v_p$ is the required power from the drive train, which is given in this case. The CVEM problem of Fig.~\ref{fig:blockdiagramseries-hybrid-powernetwork} can then be given by
	\begin{subequations}
		\begin{align}
			\min &\sum_{k\in\mathcal{K}}  y_{\text{f}} \label{eq:em-cost}\\
			\text{s.t.} & \quad \alpha_{2,m} u_{m,k}^2 + \alpha_{1,m} u_{m,k} + \alpha_{0,m} - y_{m,k} = 0 , \label{eq:em-h}\\
            & \quad x_{k+1} = x_k + \delta_Tu_{s,k}, \label{eq:em-dyn}\\
			& \quad y_{\text{f},k} + y_{\text{em},k} + s_{\text{br},k} = v_{p,k}, \label{eq:em-netw1}\\
			& \quad y_{\text{s},k} + u_{\text{em},k} = 0, \label{eq:em-netw2}
		\end{align}
	\end{subequations}
	with $\delta_{T}$ is the sampling time, and $m \in \{\text{f}, \text{s}, \text{em}\}$. Moreover, to fit the problem in the framework, the brake power becomes the slack variable for dissipating energy from the network, such that $u_{\textrm{br},k} = s_k \geqslant0$. Here, \eqref{eq:em-h} is the converter model, \eqref{eq:em-dyn} are the dynamics, and \eqref{eq:em-netw1} and \eqref{eq:em-netw2} are the network constraints.

	\begin{figure}[t]
		\centering
		\includegraphics[width=1\linewidth]{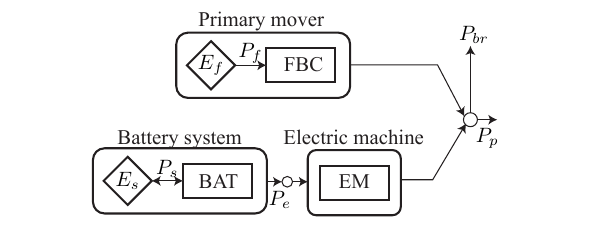}
		\caption[Block Diagram Power Network Series Hybrid Vehicle]{Block diagram power network series hybrid vehicle}
		\label{fig:blockdiagramseries-hybrid-powernetwork}
	\end{figure}
 
\end{exmp}

\begin{exmp}\label{exmp:pH}
While the previous example covered a power/energy-based problem formulation, it is also possible to model energy management problems in a port-Hamiltonian manner, where each converter has inputs and outputs which often represent voltages and currents, see \cite{padilla2020port}. In this framework, each subsystem has its internal energy expressed as,
	\begin{equation}
		\mathcal{H}_m(x_m) = \tfrac{1}{2}x^\top_{m}Q_mx_m,
	\end{equation}
	where the change in energy
	\begin{equation}
		\Delta\mathcal{H} = \mathcal{H}_m(x_{m,K}) - \mathcal{H}_m(x_{m,0})
	\end{equation} can be used in the cost function in order to penalize dissipating systems. The remainder of the formulation is represented as an input-state-output port-Hamiltonian system with continuous time dynamics,
    \begin{subequations}\label{eq:pH_dyn}
	\begin{align}
		\dot x_m &\!= \!J(x_m)\tfrac{\partial H_m}{\partial x_m} + b(x_m)u_m \label{eq:pH1}\\
		y_{m} &\!= \!b^\top(x_m)\tfrac{\partial H_m}{\partial x_m}(x_m),\label{eq:pH2}
	\end{align}
	\end{subequations}
    where $J(x_m)$ is a skew symmetric matrix called the interconnection matrix and $b(x_m)$ is the input matrix. Moreover, the network is (fully) connected through the network nodes 
	\begin{subequations}
    \begin{align}
		\sum_{m\in \mathcal{M}} f_{j,m}u_m + g_{j,m}y_m &= 0, \label{eq:pH3a} \\
        f_{j,m}y_m + g_{j,m}u_m &= 0, \quad \text{for all} \  m\in\mathcal{M} \label{eq:pH3b}
	\end{align}
    \end{subequations}
	where depending on if the node is an additive, in case of \eqref{eq:pH3a}, or equality node, i.e., \eqref{eq:pH3b}, $f_{j,m}, g_{j,m} \in \{0,1\}$ depending on if the system $m$ is connected to the node $j$. Moreover, energy conserving converters, such as dc-dc converters, are modelled as
	\begin{equation}
		y_{j,k}^\top u_{j,k} + y_{j+1,k}^\top u_{j+1,k} = 0,\label{eq:pH4}
	\end{equation}
    which can be decomposed as two converter models and one additional power node.
    
    In order to model this example into the canonical form, \eqref{eq:nonconvex-opt}, the cost function needs to be converted to a linear function, which is possible by adding another decision variable and constraint. Moreover, \eqref{eq:pH_dyn} is linearized and discretized, such that the optimal control problem becomes
    \begin{subequations}
        \begin{align}
            \min  &\sum_{m\in\mathcal{M}} \sum_{k\in\mathcal{K}} y_{m,k}  \\
            \mathrm{s.t.} \ & \tfrac{1}{2} x_{m,k+1}^\top Q_m x_{m,k+1} - \tfrac{1}{2} x_{m,k}^\top Q_m x_{m,k} - y_{m,k} = 0 \label{eq:ph-h1}\\
            & y_{m,k} u_{m,k} = u_{m+j,k} \label{eq:ph-h2}\\
            & x_{m,k+1} = A x_{m,k} + Bu_k \label{eq:ph-dyn}\\
            & y_{m,k} = C x_{m,k} + D u_{m,k} \label{eq:ph-netw1}\\
            & Fu + Gy = 0 \label{eq:ph-netw2}\\            
            & \underline{x} \leqslant x_k \leqslant \overline{x}, \quad \underline{u} \leqslant u_k \leqslant \overline{u},
        \end{align}
    \end{subequations}
    with $m \in \mathcal{M}$ are the converters in the network. Finally, \eqref{eq:ph-h1} and \eqref{eq:ph-h2} are the converters, \eqref{eq:ph-dyn} are the dynamics and \eqref{eq:ph-netw1} and \eqref{eq:ph-netw2} are the network constraints.
\end{exmp}

It can be seen that the modelling framework, consisting of buffers, converters and network nodes, can be used to formulate all these examples, thus allowing the examples to be formulated in the form of Problem \eqref{eq:nonconvex-opt}. An example of a more complex problem with more converters has been studied in \cite{romijn2018distributed} and a problem with more complex buffer models (in Wiener-Hammerstein form) have been studied in \cite{padilla2022complete} and also fit the form of Problem \eqref{eq:nonconvex-opt}. The next section will focus on finding a solution to the posed problems.

\section{Solution Framework}\label{sec:solution-framework}

As motivated in the previous section, the modelling framework is sufficiently general. Still, the optimization problem \eqref{eq:nonconvex-opt} is non-convex because of the nonlinear energy converter \eqref{eq:nonconvex-opt-h}. In this section, we will present a solution strategy based on a convex relaxation of \eqref{eq:nonconvex-opt} and show that this convex problem has the same (global) solution as \eqref{eq:nonconvex-opt}, thereby generalizing the results of \cite{romijn2018distributed,padilla2022complete,egardt2014electromobility}. 

\begin{figure}[t]
\centering
\vspace{2mm}
\includegraphics[width=\linewidth]{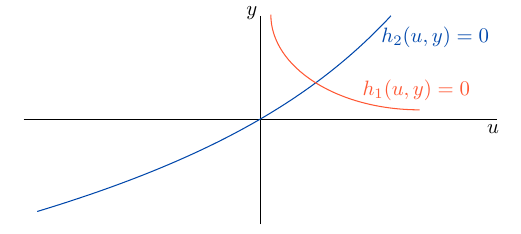}
\caption{Two common forms of convertor models $h_m = 0$. It can be seen that for $h_1$, we have $\frac{\partial y}{\partial u}\geqslant 0$ and $\frac{\partial y}{\partial u}\geqslant 0$ for $h_2$.}
\label{fig:functions}
\end{figure}

For the convex relaxation to have the same solution, we require that the converters models
\begin{itemize}
\item[\emph{i)}] can be relaxed to a convex function, i.e., $\nabla^2 h_m(x,u,y) \succeq 0$
\item[\emph{ii)}] are strictly decreasing in its output, i.e., $\frac{\partial h_m(x,u,y)}{\partial y} < 0$
\end{itemize}
for all feasible $x$, $u$ and $y$. Furthermore, we require the power network to be such that 
\begin{itemize}
\item[\emph{iii)}] the entries of $b_m$ and the matrix $G$ are positive
\item[\emph{iv)}] if an output of a converter $y_{m,k}$ is connected to the node, the states $x_{m,k}$ and inputs $u_{m,k}$ cannot be connected to that node
\item[\emph{v)}] $\mathrm{rank}\big(F + G \mathrm{diag}(\frac{\partial y_{1,k}}{ \partial u_{1,k}},\ldots, 
\frac{\partial y_{M,k}}{ \partial u_{M,k}}) \big) = J$, for all feasible $x_{m,k}$, $u_{m,k}$ and $y_{m,k}$, with $J$ the number of network nodes, which corresponds to the number of rows in $E$, $F$ and $G$, and $\textrm{diag}(\cdot)$ is a block-diagonal matrix with its arguments on the block diagonal. 
\end{itemize}
It should be noted that these requirements are mild requirements. Namely, Requirement \emph{i)} enables a convex relaxation, while Requirements \emph{ii)} and Requirements \emph{v)} are related to monotonicity, where $\frac{\partial y_{m,k}}{ \partial u_{m,k}} = - (\frac{\partial h_m}{ \partial y_{m,k}})^{-1} \frac{\partial h_{m}}{ \partial u_{m,k}}$. Examples of such functions are given in Fig.~\ref{fig:functions}. Finally, self-loops, as excluded by Requirement \emph{iv)} would make a converter redundant, and implies that the entries of $E$ and $G$, and $F$ and $G$, cannot have the same nonzero entries, which could violate Requirement \emph{v)}.

\subsection{Global Solutions using Convex Relaxations}

Since the only non-convex constraint in \eqref{eq:nonconvex-opt} is $h_m(x,u,y)=0$, we relax this constraint into an inequality leading to
\begin{subequations}\label{eq:NLLP_rel}
\begin{align}
\min_{y_{k},u_{k},x_{k},s_k} & \sum_{k\in\mathcal{K}} \!  a^\top u_{k} + b^\top y_{k} \label{eq:NLLP_rel_a}\\
\text{s.t.}\quad & h_m(x_{m,k},u_{h,k},y_{m,k}) \leqslant 0 \label{eq:NLLP_rel_b}\\
                 & E x_{k} +  F u_{k} + G y_k + s_{k}=v_k, \label{eq:NLLP_rel_c} \\
                 & x_{k+1}=A x_{k} + B u_{k} + f_k \label{eq:NLLP_rel_d}\\
                 &\underline{x} \leqslant x_k \leqslant \overline{x}, \ \ \underline{u} \leqslant u_k \leqslant \overline{u}, \ \ s_k\geqslant0 \label{eq:NLLP_rel_e}
\end{align} 
\end{subequations}
with $a=[a_1^\top \ldots a_M^\top]^\top$ and $b=[b_1 \ldots b_M]^\top$, and where some elements of $s_k$ are zero, namely for the energy conserving nodes. It should be noted that because of Requirement \emph{i)}, the above problem is convex. This means that it can be solved efficiently and that all solutions are global solutions. 

The question remains whether solutions to \eqref{eq:NLLP_rel} are also solutions to \eqref{eq:nonconvex-opt}, or at least how to obtain solutions to \eqref{eq:nonconvex-opt} using \eqref{eq:NLLP_rel}. To address this question, let us consider the fact that the solutions to the above optimization problems satisfy
\begin{equation}
d_{\mathrm{CR}} \leqslant p_{\mathrm{CR}} \leqslant p_{\mathrm{NC}},
\end{equation}
where $p_{\mathrm{NC}}$ is the (primal) optimal value of \eqref{eq:nonconvex-opt}, and $p_{\mathrm{CR}}$ and 
$d_{\mathrm{CR}}$ are the primal and dual optimal value of \eqref{eq:NLLP_rel}, respectively, see \cite{boyd2004convex}. The fact that $p_{\mathrm{CR}} \leqslant p_{\mathrm{NC}}$ follows from the fact that the feasible set of the convex relaxation is larger than the original problem. However, we will show below that $p_{\mathrm{CR}} = p_{\mathrm{NC}}$ and we need $d_{\mathrm{CR}} = p_{\mathrm{CR}}$ in its proof, for which we need the constraints to satisfy some form of constraint qualification, see \cite{boyd2004convex}. The following lemma will also generalize Lemma 1 of \cite{padilla2022complete}.

\begin{lem}
Optimization problem \eqref{eq:NLLP_rel} satisfies linear independence constraint qualification (LICQ), leading to $d_{\mathrm{CR}} = p_{\mathrm{CR}}$, if Requirements \emph{ii,iv,v)} are satisfied.
\end{lem}

\begin{proof}
LICQ holds if the gradient of all the active constraints are linearly independent, see \cite{boyd2004convex}. If we define $\omega = [ x^\top \ u^\top \ y^\top ]^\top$, with $x = [x_0^\top \ \ldots \ x_{K-1}^\top]^\top$, and $u$ and $y$ defined similarly,
collect the constraints (\ref{eq:NLLP_rel}b) as $h(x_{k},u_{k}, y_{k})=[
h_1(x_{1,k},u_{1,k}, y_{1,k}) \ldots h_M(x_{M,k},u_{M,k}, y_{M,k})]^\top$ 
and collect the constraints (\ref{eq:NLLP_rel}c) as $\Gamma_x [x^\top x_{K}^\top ]^\top + \Gamma_u u = [ x_{\mathrm{init}}^\top \ f_1^\top \ f_2^\top \ \ldots ]^\top$, with
\begin{equation}
\Gamma_x = {\footnotesize\begin{bmatrix} I &  0 & \ldots & 0  \\ -A & I & \ddots  & \vdots \\ 0  & \ddots & \ddots & 0  \\   & 0  & -A & I \end{bmatrix}}, \quad
\Gamma_u = {\footnotesize\begin{bmatrix} 0 & \dots & 0 \\ -B & 0 & \vdots \\ 0 & \ddots  & 0 \\ 0 & \ldots  & -B\end{bmatrix}}
\end{equation}
the gradient of constraints of (\ref{eq:NLLP_rel}b-d) are given by
\begin{equation}\label{eq:LICQ}
\Xi(x,u,y) =  \begin{bmatrix}
I \otimes \frac{\partial h}{\partial x_k} & I \otimes \frac{\partial h}{\partial u_k} & I \otimes \frac{\partial h}{\partial y_k} \\
\Gamma_x & \Gamma_u & 0 \\
I \otimes E & I \otimes F & I \otimes G
\end{bmatrix}.
\end{equation}
We excluded constraints (\ref{eq:NLLP_rel}e) as they are always linearly independent. It should be noted that $\frac{\partial h}{\partial x_k}$, $\frac{\partial h}{\partial u_k}$, and $\frac{\partial h}{\partial y_k}$ are block-diagonal. Now because $\Gamma_x$ and $\frac{\partial h}{\partial y_k}$ are both invertible (the latter because of Requirement \emph{ii)}), we can premultiply \eqref{eq:LICQ} as follows:
\begin{align}
\begin{bmatrix} I\!\otimes\!(\frac{\partial h}{\partial y_k})^{-1} & 0 & 0 \\ 0 & \Gamma_x^{-1} & 0 \\ -I\!\otimes\!G (\frac{\partial h}{\partial y_k})^{-1} & \!\!(I \!\otimes\! (G \frac{\partial y}{\partial x} - E)) \Gamma_x^{-1} & \!\!I
\end{bmatrix}
\ \Xi(x,u,y) \notag \\ 
=  \begin{bmatrix} - I \!\otimes\! \frac{\partial y}{\partial x} & - I\!\otimes\!\frac{\partial y}{\partial u} & I \\ I & \Gamma_x^{-1} \Gamma_u & 0 \\ 0 & \Psi & 0 
\end{bmatrix}
\end{align}
with implicit derivatives $\frac{\partial y}{\partial x} = - (\frac{\partial h}{\partial y})^{-1} \frac{\partial h}{\partial x}$, $\frac{\partial y}{\partial u} = - (\frac{\partial h}{\partial y})^{-1} \frac{\partial h}{\partial u}$, which are block diagonal, and
\begin{equation}\label{eq:Psi}
\Psi = I \otimes ( F + G \tfrac{\partial y}{\partial u}) - (I \!\otimes\! (E + G \tfrac{\partial y}{\partial x})) \Gamma_x^{-1} \Gamma_u.
\end{equation}
We now conclude that \eqref{eq:LICQ} loses rank, only when \eqref{eq:Psi} looses rank. Because $\Gamma_x^{-1} \Gamma_u$ is strictly lower-triangular (with zeros on the block diagonal), \eqref{eq:Psi} and the other elements are block-diagonal, \eqref{eq:Psi} can only loose rank when the block-diagonal elements lose rank, i.e., when $F + G \tfrac{\partial y}{\partial u}$ is not full row rank. Therefore, LICQ is satisfied when Requirement \emph{v)} is met, which completes the proof.
\end{proof}


The main result of this section shows that $p_{\mathrm{CR}} = p_{\mathrm{NC}}$, in which the dual optimization problem of \eqref{eq:NLLP_rel} is used and generalizes \cite{romijn2018distributed} to beyond quadratic converter models.

\begin{thm}\label{th}
Assume that Requirements \emph{i-v)} are satisfied and, thus, Problem \eqref{eq:nonconvex-opt} satisfies LICQ. Then, the solution to \eqref{eq:NLLP_rel} satisfies (\ref{eq:NLLP_rel}b) with equality, i.e., it solves \eqref{eq:nonconvex-opt}, or a solution to \eqref{eq:NLLP_rel} exists that satisfies \eqref{eq:NLLP_rel_b} with equality. Either way, $p_{\mathrm{CR}} = p_{\mathrm{NC}}$.
\end{thm}

\begin{proof}
Because of LICQ, we have that $d_{\mathrm{CR}} = p_{\mathrm{CR}}$ and we can consider the dual problem
\begin{align}\label{eq:dual}
\!\!\max_{\mu_k\geqslant0,\lambda_k}\!\min_{x_k,u_k,y_k} \sum_{k\in\mathcal{K}}&a^{\!\top} u\!+\!b^{\!\top} y 
+\mu_k^\top h(x_{k},u_{k}, y_{k}) \notag \\[-5pt]
&+\!\lambda_k^{\!\top} (Ex_k\!+\!Fu_k\!+\!Gy_k\!+\!s_k\!-v_k ) \notag \\
&\qquad\qquad+ i_\Omega(x_k,u_k,s_k)
\end{align}
where $i_\Omega(x_k,u_k,s_k)$ denotes the indicator function for constraints (\ref{eq:NLLP_rel}d,e) and $h(x_{k},u_{k}, y_{k})=[
h_1(x_{1,k},u_{1,k}, y_{1,k}) \ldots h_M(x_{M,k},u_{M,k}, y_{M,k})]^\top$. The stationary condition with respect to $y_k$ (one of the conditions for optimality) is given by
\begin{equation}
b^\top + G^\top \lambda_k + \tfrac{\partial h(x_k,u_k,y_k)}{\partial y_k} \mu_k = 0
\end{equation}
Now because $b$ and $G$ only have positive elements and $\tfrac{\partial h(x,u,y)}{\partial y} <0$, we cannot have, $\lambda<0$ as it would violate the necessary condition that $\mu\geqslant0$. Therefore $b^\top + G^\top \lambda\geqslant0$, either means that rows of $b^\top + G^\top \lambda$ are positive, which means that the corresponding elements of $\mu$ are positive, thus that the constraints in (\ref{eq:NLLP_rel}b) are satisfied with equality, or rows in $b^\top + G^\top \lambda$ and the corresponding elements of $\mu$ are equal to zero. In the latter case, those terms will disappear from the dual optimization problem \eqref{eq:dual}, meaning that a $y$ can be selected so that it satisfies \eqref{eq:NLLP_rel_b} with equality. This means that $p_{\mathrm{CR}} = p_{\mathrm{NC}}$, which completes the proof.
\end{proof}

\subsection{Enforcing Exact Relaxations}

As can be concluded from Theorem \ref{th}, it is not guaranteed that \eqref{eq:NLLP_rel} provides a solution where \eqref{eq:NLLP_rel_b} is satisfied with equality, though a solution exists that does. Namely, \eqref{eq:NLLP_rel} can have multiple global minima, i.e., different solutions that all achieve the same minimal cost, of which at least one satisfies (\ref{eq:NLLP_rel}b) with equality. This can happen, for instance, when two converter models are connected to a dissipative node at the output. One of the branches will be minimized, whereas the other branch is left with too much energy. 

Since \eqref{eq:NLLP_rel} has multiple global minima, we can add regularization to ensure the global minimum is found of \eqref{eq:NLLP_rel} that is also a global minimum of \eqref{eq:nonconvex-opt}. An approach that results in the relaxed constraints being satisfied with equality is to first solve \eqref{eq:NLLP_rel}, check whether the constraints of \eqref{eq:nonconvex-opt} are satisfied. If this is not the case, add the outputs, $y_k$, that are not directly satisfied with equality to the cost function and multiply them with a small weight, and solve again using \eqref{eq:NLLP_rel}. Since the original optimization problem has only one global minimum, adding the output of a carefully chosen converter to the cost function will steer the solution towards a solution that has the smallest possible feasible $y$, i.e., having (\ref{eq:nonconvex-opt}) with equality. For Example~\ref{exmp:ED}, we can also add the output of the lethargy to the cost function and in the case of Example~\ref{exmp:EM} the output power of the EM could be added. 

\subsection{Second-order Cone Formulation}\label{sec:socp}

Even though the relaxed problem \eqref{eq:NLLP_rel} is a convex optimization problem, it still may be difficult to solve as we did not further specify the form of the function $h_m(x,u,y)$. If Requirements \emph{i-v)} are satisfied, a general nonlinear programming solver can find a global optimal solution, though the computational complexity might be large, and global optimization methods, such as multi-shooting or genetic algorithms are no longer needed. More interestingly, for some very common functions $h_m(x,u,y)$, the convex relaxations actually lead to second-order cone problems, where the feasible set is a Lorenz cone, given by
\begin{equation}
	\mathcal{Q}_N = \{\mathbf{x} = [x_0\quad \bar{x}]^\top \in \Real^N: x_0\geqslant \norm{\bar{x}}_2\},
\end{equation}
or equally $\mathbf{x} \geq_{\mathcal{Q}_N}0$, depending on the form of $h(\cdot)\in\mathcal{Q}_N$, see \cite{Alizadeh2003SOCP}. For Examples \ref{exmp:ED}, \ref{exmp:EM} and \ref{exmp:pH}, it can be seen that the modelling framework is general enough to fit these problems, as they all have linear dynamics, linear network constraints and nonlinear converters that can be written as second-order conic constraints by relaxing the equality into an inequality.

The non-linearities found in Examples \ref{exmp:ED}-\ref{exmp:pH} are very common, and can all expressed as quadratic functions of the form
\begin{equation}
\tfrac{1}{2} \xi^\top Q \xi + a^\top \xi + \beta = 0, \label{eq:quadraticconstraint}
\end{equation}
with $\xi = [x^\top \ u^\top \ y ]^\top$. For a linear function, we have that $Q = 0$ and this leads to a polyhedral constraint. This enables the use of variations of the following converter models
\begin{align}
\tfrac{1}{2\varepsilon} u_{m,k}^2 + u_{m,k} - y_{m,k} &= 0 \\ 
y_{m,k}(u_{m,k}+\varepsilon) &= 1 
\end{align}
with either $\underline{u}_m > -\varepsilon$ or $\overline{u}_m < -\varepsilon$, which ensure monotonicity, as needed for Requirement \emph{v)}. All these conditions can be rewritten as a second-order conic constraint ~\cite{lobo1998applications}, leading to a SOCP for which efficient solvers exist.

\section{Simulation study showing the exactness of relaxations}\label{sec:numex}

This section contains a simulation study of Example \ref{exmp:ED} with values from Table~\ref{tab:my_label}. The nonlinear constraints are relaxed into second-order conic constraints as proposed in \eqref{eq:NLLP_rel_e}. For the simulation, we consider three cases, namely:
\begin{itemize}
    \item Case 1, where an upper bound of $T_{\max}= 700$ seconds is given to complete the route.
    \item Case 2, in which the upper time limit is set to a large number ($T_{\max} = 10^5$ seconds in this case)
    \item Case 3, where in addition to case 2 the lethargy is added to the cost function as regularization, resulting in the cost function $\min \sum_{k\in \mathcal{K}} y_{em,k} + \sigma y_{\ell,k}$, where $\sigma$ is chosen to be $0.01$.
\end{itemize}

First, we show that the example satisfies the requirements given in Section~\ref{sec:solution-framework}, since, the converter models can be relaxed into convex functions (i), both \eqref{eq:mv2} and \eqref{eq:lethv} are strictly decreasing in the output (ii) and the network requirements \emph{iii)-iv)} are satisfied. Requirement \emph{v)} is also satisfied as can be shown by considering the network matrices,
$$F = \begin{bmatrix}
    0 & 0 & 0 \\ 0 & 1 & 0
\end{bmatrix} \text{ and } G = \begin{bmatrix}
    0 & 1 & 0 \\ 0 & 0 & 1
\end{bmatrix},$$
with
$$u_k \!=\!\begin{bmatrix}
    u_{\textrm{kin},k} & u_{\textrm{v},k} & u_{\ell,k}
\end{bmatrix}^\top\!\! \text{ and } \ y_k \!=\! \begin{bmatrix}
    y_{\textrm{kin},k} & y_{\textrm{v},k} & y_{\ell,k}
\end{bmatrix}^\top\!\!,$$
the partial derivatives of the converter models,
\begin{equation*}
    \mathrm{diag}(\frac{y_{1,k}}{u_{1,k}}, \dots, \frac{y_{M,k}}{u_{M,k}}) = \begin{bmatrix}
        0 & 0 & 0 \\
        0 & u_{\mathrm{v},k} & 0 \\
        0 & 0 & -\frac{1}{u_{\ell,k}}
    \end{bmatrix}
\end{equation*}
and substituting these into Requirement \emph{v)}, which results into the rank condition being fulfilled for $\underline{u}_{\mathrm{v},k}>0$ and $\underline{u}_{\ell,k} > 0$.

In every case, the optimal control problem can be solved by MOSEK with an average computation time of $0.55$, $0.49$ and  $0.51$ seconds for Case 1 to 3, respectively. These computations times are achieved using a horizon length of 2500 on a pc with an Intel Core i7-9750H at 2.6Ghz and 8Gb of RAM.
\begin{table}[b]
	\centering
	\caption{Parameters and values for numerical example.}
	\begin{tabular}{|l|c|c|}
		\hline
		symbol & Parameter & Value \\ \hline
		$m$ & mass & 13,400 kg \\
		$\delta_s$ & Sampling distance & 5 m \\
		$A$ & & 0.9981 \\
		$B$ & & 0.005 \\
		\hline        
	\end{tabular}
	
	\label{tab:my_label}
\end{table}
The corresponding solution to the three cases can be seen in Fig.~\ref{fig:highvel}, which shows the produced velocity profile, and Fig.~\ref{fig:vehpropulsion}, which shows the required propulsion force to reach each of the velocity profiles. The residual for the relaxed converter models are plotted in Fig.~\ref{fig:errors}. By examining these residuals, it can be seen that Case 1 has residuals in the order of $10^{-7}$ to $10^{-4}$, which is in most cases enough for the solution to be exact with respect to the non-convex formulation. However, in Case 2 it is clear that the magnitude of the residual is much higher than in Case 1. This can be attributed to the fact that after 3000 meters, the vehicle starts rolling down the hill, which increases the kinetic energy, but is not properly propagated through the network. In Case 3, the same simulation is performed, but now the lethargy is added to the cost function. Looking at the residuals in Fig.~\ref{fig:errors}, shows that the magnitude of the error has decreased 250 times. This change is also reflected in the velocity profile in Fig.~\ref{fig:highvel}, where a slightly higher velocity can be observed. From this, it can be concluded that adding a regularization term to the cost function, leads to \eqref{eq:NLLP_rel_b} in our solution framework satisfying the equality in \eqref{eq:nonconvex-opt-h}, which shows that the problem is solved using an exact relaxation. Moreover, it can be seen that the input forces in Fig.~\ref{fig:vehpropulsion} are indistinguishable between Case 2 and 3, which means that although the cost function was changed, the optimal cost has remained the same. 

\begin{figure}
    \vspace{2mm}
	\centering
	\includegraphics{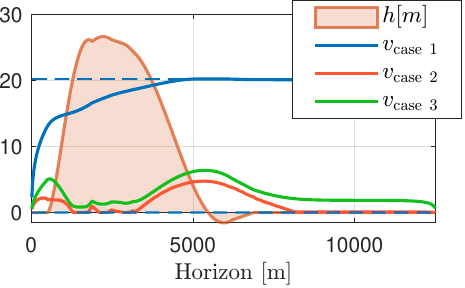}
	\caption{Velocity profile resulting from solving the eco-driving problem subject to a given height map. The velocity is given in $m/s$.}
	\label{fig:highvel}
\end{figure}

\begin{figure}
	\centering
	\includegraphics{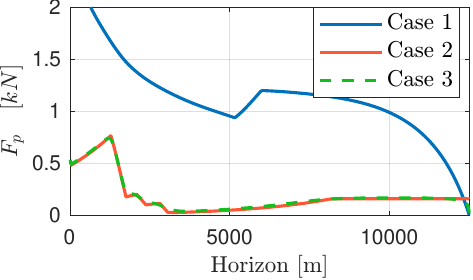}
	\caption{Applied propulsion force corresponding to the velocity profiles in Fig.~\ref{fig:highvel}.}
	\label{fig:vehpropulsion}
\end{figure}

\begin{figure}
    \centering
    \begin{subfigure}{0.5\linewidth}
    \centering
        \includegraphics{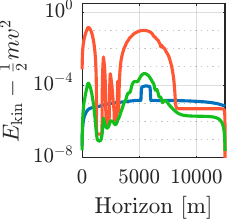}
    \end{subfigure}%
    \begin{subfigure}{0.5\linewidth}
        \centering
        \includegraphics{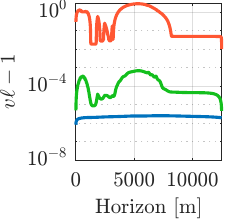}
    \end{subfigure}
    \caption{Relative error of the two non-linear functions given for the cases in Fig.~\ref{fig:highvel}}\label{fig:errors}
\end{figure}

\section{Conclusion}\label{sec:conclusions}
In this paper, we have generalized multiple methods for solving energy management problems into a general framework. We have shown that a large variety of problems and models can be used within the framework, as illustrated by the examples in this paper. Moreover, the resulting non-convex optimization problem can be relaxed into a convex optimization problem. Conditions are given such that the resulting solution is also the solution of the original non-convex problem. A mathematical proof of this property of the framework was given by using linear independence constraint qualification and optimality conditions of the relaxed problem. Moreover, we show that certain converter models can be solved using second-order cone programs, which means that off the shelf solvers can be used to solve the problem.

\balance
\bibliographystyle{IEEEtran}
\bibliography{references.bib}
\end{document}